\documentclass[12pt,psamsfonts]{amsart}
\usepackage{amsmath}
\usepackage{amssymb}
\usepackage{amsthm}
\usepackage{mathrsfs}
\usepackage[cp1252]{inputenc}
\usepackage{hyperref}
\usepackage{fancyhdr}
\theoremstyle{plain}
\newtheorem{theo}{Theorem}[section]
\newtheorem{pr}{Proposition}[section]

\newtheorem{lem}{Lemma}[section]
\theoremstyle{remark}
\newtheorem*{rem}{Remark}
\newtheorem*{nota}{Notation}

\theoremstyle{definition}

\numberwithin{equation}{section}
\begin{document}
\bibliographystyle{plain}
\title{Two results on the Dunkl maximal operator}
\author{Luc DELEAVAL}
\date{July, 2010}
\address{Institut de Mathématiques de Jussieu, Université Pierre et Marie Curie, 175, rue du Chevaleret, 75013 Paris, France}
\email{deleaval@math.jussieu.fr}
\thanks{The author is pleased to express his thanks to Banafsheh for her advises in english.}
\begin{abstract} 
In this article, we first improve the scalar maximal theorem for the Dunkl maximal operator by giving  some precisions on the behavior of the constants of this theorem for a general reflection group. Next we complete the vector-valued theorem for the Dunkl-type Fefferman-Stein operator in the case $\mathbb Z_2^d$ by establishing a result of exponential integrability corresponding to the case $p=+\infty$.
\end{abstract}
\maketitle

\section{Introduction}

Dunkl operators provide an essential tool to extend Fourier analysis on Euclidean spaces and analysis on Riemannian symmetric spaces of Euclidean type. Since their invention in 1989, these operators have largely contributed, in the setting of root systems and associated reflection groups, to the development of harmonic analysis and to the theory of multivariable hypergeometric functions. \newline \indent In this paper, we focus on the Dunkl maximal operator $M_\kappa^W$ which is defined by
\[
M_\kappa^Wf(x)=\sup_{r>0}\frac{1}{\mu_\kappa^W(B_r)}\biggl|\int_{\mathbb R^d}f(y)\tau_x^{W}(\chi_{{}_{B_r}})(-y)\,\mathrm{d}\mu_\kappa^W(y)\biggr|, \quad x \in \mathbb R^d,
\]
where $\chi_{{}_{B_r}}$ is the characteristic function of the Euclidean ball of radius $r$ centered at the origin, $\tau_x^{W}$ is the Dunkl translation and $\mu_\kappa^W$ is a weighted Lebesgue measure invariant under the action of the reflection group $W$ (see Section 2 for more details). This operator, which reduces to the well-known Hardy-Littlewood maximal operator in the case where the multiplicity function $\kappa$ is equal to $0$ (see Section 2 for details), is of particular interest for harmonic analysis associated with root systems. Nevertheless, the structure of the Dunkl translation prevents us from using the tools of real analysis (covering lemma, weighted inequality, Calder\'on-Zygmund decomposition...) and makes  the study of $M_\kappa^W$ difficult. \newline \indent However, Thangavelu and Xu  \cite{ThangaXuMaximal} succeeded in proving  the following scalar maximal theorem, where we denote by $L^p(\mu_\kappa^W)$ the space $L^p(\mathbb R^d; \mu_\kappa^W)$ (for $1\leqslant p\leqslant +\infty$) and we use the shorter notation $\mathopen\|\cdot\mathclose\|_{W,p}$ instead of $\mathopen\|\cdot\|_{L^p(\mu_\kappa^W)}$.

\begin{theo}[\bf{Scalar maximal theorem}] \label{TheoMaxXu}
 Let $f$ be a measurable function defined on $\mathbb R^d$.
\begin{enumerate}
\item If $f \in L^1(\mu_\kappa^W)$, then for every  $\lambda>0$, 
\[
\mu_\kappa^W\Bigl(\Bigl\{x \in \mathbb R^d: M_\kappa^Wf(x)>\lambda\Bigr\}\Bigr)\leqslant\frac{C}{\lambda}\|f\|_{W,1},
\]
where $C=C(d,\kappa)$ is a constant independent of $f$ and $\lambda$.
\item If $f \in L^p(\mu_\kappa^W)$ with $1<p\leqslant+\infty$, then $M_\kappa^Wf \in L^p(\mu_\kappa^W)$ and 
\[
\bigl\|M_\kappa^Wf\bigr\|_{W,p}\leqslant C\|f\|_{W,p},
\]
where $C=C(d,\kappa,p)$ is a constant independent of $f$.
\end{enumerate}
\end{theo}

In order to prove this theorem, Thangavelu and Xu have used the following Hopf-Dunford-Schwartz ergodic theorem (see \cite{DunfordSchwartz}).

\begin{theo}[\bf{Hopf-Dunford-Schwartz ergodic theorem}] \label{HDS}
 Let $X$ be a measurable space and let $m$ be a positive measure on $X$. Let $\{T_t\}_{t\geqslant0}$ be a  contraction semigroup of operators on $L^p(X;m)$, that is,  a semigroup which satisfies, for every $p \in [1,+\infty]$ and every $f \in L^p(X;m)$,
\[
\|T_tf\|_{L^p(X;m)}\leqslant\|f\|_{L^p(X;m)}.
\]
Define
\[
Mf(x)=\sup_{t>0}\biggl|\frac{1}{t}\int_{0}^{t}T_{s}f(x)\,\mathrm{d}s\biggr|. 
\]
\begin{enumerate}
\item If $f \in L^1(X;m)$, then for every $\lambda>0$, 
\[
m\Bigl(\Bigl\{x \in X: Mf(x)>\lambda\Bigr\}\Bigr)\leqslant\frac{2}{\lambda}\|f\|_{L^1(X;m)}.
\]
\item If $f \in L^p(X;m)$, with $1<p\leqslant+\infty$, then $Mf \in L^p(X;m)$ and 
\[
\bigl\|Mf\bigr\|_{L^p(X;m)}\leqslant C\|f\|_{L^p(X;m)},
\]
where $C=C(p)$ is a constant independent of $f$.
\end{enumerate}
\end{theo}

We will see that we can use  the previous theorem in order to refine the scalar maximal theorem. More precisely, our first result is the following, where we denote by $2\gamma$ the degree of homogeneity of the measure $\mu_\kappa^W$.

\begin{theo} \label{firstresult} Let $f$ be a measurable function defined on $\mathbb R^d$.
\begin{enumerate}
\item There exists a numerical constant $C$ such that if $f \in L^1(\mu_\kappa^W)$, then  for every $\lambda>0$,
\[
\mu_\kappa^W\Bigl(\Bigl\{x \in \mathbb{R}^d: M_\kappa^Wf(x)>\lambda\Bigr\}\Bigr)\leqslant C\frac{d+2\gamma}{\lambda}\|f\|_{W,1}.
\]
\item There exists a numerical constant $C$ such that if $f \in L^p(\mu_\kappa^W)$ with $1<p<+\infty$, then 
\[
\bigl\|M_\kappa^Wf\bigr\|_{W,p}\leqslant C\frac{p}{p-1}\sqrt{d+2\gamma}\,\|f\|_{W,p}.
\]
\end{enumerate}
\end{theo}

In the particular case where $\gamma=0$, the previous theorem coincides with a theorem due to Stein and Strömberg  for the Hardy-Littlewood maximal operator (\cite{SteinStromberg}).

Our second result deals with the vector-valued extension of the scalar maximal theorem which has been proved in \cite{Luc} in the case where  the reflection group is $\mathbb Z_2^d$. Let us recall this theorem. Denote by $\mathcal M_\kappa^{\mathbb Z_2^d}$ the Dunkl-type Fefferman-Stein operator given for a sequence $f=(f_n)_{n\geqslant1}$ of measurable functions by
\[
\mathcal M_\kappa^{\mathbb Z_2^d}f=\bigl(M_\kappa^{\mathbb Z_2^d}f_n\bigr)_{n\geqslant1}.
\]

\begin{theo}[\bf{Vector-valued maximal theorem}] \label{FS}
Let $W=\mathbb Z_2^d$ and let $f=(f_n)_{n\geqslant1}$ be a sequence of measurable functions defined on $\mathbb R^d$. 
\begin{enumerate} 
\item Let $1<r<+\infty$. If $\|f\|_{\ell^r} \in L^1(\mu^{\mathbb Z^d_2}_\kappa)$, then for every $\lambda>0$, 
\[
\mu^{\mathbb Z^d_2}_\kappa\Bigl(\Bigl\{x \in \mathbb R^d: \|\mathcal M_\kappa^{\mathbb Z_2^d}(x)\|_{\ell^r}>\lambda\Bigr\} \Bigr)\leqslant \frac{C}{\lambda}\bigl\|\,\|f\|_{\ell^r}\bigr\|_{\mathbb Z^d_2,1},
\]
where $C=C(d,\kappa,r)$ is a constant independent of $(f_n)_{n\geqslant1}$ and $\lambda$.
\item Let $1<r,p<+\infty$. If $\|f\|_{\ell^r} \in L^p(\mu^{\mathbb Z^d_2}_\kappa)$, then 
\[
\bigl\|\,\|\mathcal M_\kappa^{\mathbb Z_2^d}f\|_{\ell^r}\bigr\|_{\mathbb Z_2^d,p}\leqslant C\bigl\|\,\|f\|_{\ell^r}\bigr\|_{\mathbb Z_2^d,p},
\]
where $C=C(d,\kappa,p,r)$ is a constant independent of $(f_n)_{n\geqslant1}$.
\end{enumerate}
\end{theo}

We will see that no analogue of $(2)$ holds when $p=+\infty$. However,  in this case we will give the following result of exponential integrability on every compact set, which generalizes the classical one due to Fefferman and Stein (see \cite{FeffermanStein} or \cite[page 75]{SteinGros}).

\begin{theo} \label{expint}
Let $f=(f_n)_{n\geqslant1}$ be a sequence of measurable functions defined on $\mathbb{R}^d$ and let $1<r<+\infty$. If  $\|f\|_{\ell^r} \in L^\infty(\mu^{\mathbb{Z}^d_2}_\kappa)$ is such that
\[
\mu^{\mathbb{Z}^d_2}_\kappa\Bigl(\mathrm{supp }\|f\|_{\ell^r}^r\Bigr)<+\infty, 
\]
then the function $\|\mathcal{M}_\kappa^{\mathbb{Z}_2^d}f\|_{\ell^r}^r$ is exponentially integrable on every compact set. More precisely there exists a constant $C_{d,\kappa,r}$, which depends only on $d$, $\kappa$ and $r$, such that for every compact subset $K$ of  $\mathbb{R}^d$ and for every $\varepsilon$ satisfying
\[
0\leqslant\varepsilon<\frac{\log( 2)}{2C_{d,\kappa,r}\bigl\|\,\|f\|_{\ell^r}^r\bigr\|_{\mathbb{Z}_2^d,\infty}}, 
\]
we have the inequality
\begin{multline*}
\int_K\mathrm{e}^{\varepsilon\|\mathcal{M}_\kappa^{\mathbb{Z}_2^d}f(x)\|_{\ell^r}^r}\,\mathrm{d}\mu_\kappa^{\mathbb{Z}_2^d}(x)\\
\leqslant \mu^{\mathbb{Z}^d_2}_\kappa(K)+\frac{2\varepsilon C_{d,\kappa,r}\bigl\|\,\|f\|_{\ell^r}^r\bigr\|_{\mathbb{Z}_2^d,\infty}\max\Bigl\{2\mu^{\mathbb{Z}^d_2}_\kappa(K);\mu^{\mathbb{Z}^d_2}_\kappa\Bigl(\mathrm{supp }\|f\|_{\ell^r}^r\Bigr)\Bigr\}}{\log (2)-2\varepsilon C_{d,\kappa,r}\bigl\|\,\|f\|_{\ell^r}^r\bigr\|_{\mathbb{Z}_2^d,\infty}}. 
\end{multline*}
\end{theo}

The paper is organized as follows. In the next section, we collect some definitions and results related to Dunkl’s analysis. We then give in Section 3 the proof of Theorem \ref{firstresult}. A counterexample in the case $p=+\infty$ is given for the vector-valued maximal theorem in Section 4 and the substitute result contained in Theorem \ref{expint} is established.

\section{Preliminaries}

This section is devoted to the preliminaries and background. We only focus on the aspects of the Dunkl theory which will be relevant in what follows. For a large survey about this theory, the reader may especially consult \cite{DunklXU,Roslerlecturenotes} and the references therein.
\newline \indent 
Let $W \subset \mathcal O(\mathbb{R}^d)$ be a finite reflection group associated with a reduced root system $\mathcal R$ (not necessarily crystallographic) and let $\kappa:\mathcal R\to\mathbb C$ be a multiplicity function, that is, a $W$-invariant function. We assume in this article that $\kappa$ takes value in $[0,+\infty[$.
\newline \indent 
The (rational) Dunkl operators $T_\xi^{\mathcal R}$ on $\mathbb R^d$,  introduced in \cite{Dunkl89}, are the following $\kappa$–deformations of directional derivatives $\partial_\xi$ by reflections:
\[
T_\xi^{\mathcal R}f(x)=\partial_\xi f(x)+\sum_{\alpha \in \mathcal R_+}\kappa(\alpha)\frac{f(x)-f( \sigma_\alpha(x))}{\langle x,\alpha\rangle}\langle\xi,\alpha\rangle,\quad x \in \mathbb R^d,
\]
where $\sigma_\alpha$ denotes the reflection with respect to the hyperplane orthogonal to $\alpha$, and $\mathcal R_+$ denotes a positive subsystem of $\mathcal R$. The definition is of course independent of the choice of a positive subsystem since $\kappa$ is $W$-invariant. The most important property of these operators is their commutativity, that is, $T^{\mathcal R}_{\xi}T^{\mathcal R}_{\xi'}=T^{\mathcal R}_{\xi'}T^{\mathcal R}_{\xi}$ (\cite{Dunkl89}). Therefore, we are naturally led to consider the eigenfunction problem 
\begin{equation} \label{spectral} T^{\mathcal R}_\xi f=\langle y,\xi\rangle f \quad \forall \xi \in \mathbb R^d\end{equation}
with $y \in \mathbb C^d$ a fixed parameter. This problem has been completely solved by Opdam (\cite{Opdam}).

\begin{theo}
Let $y \in \mathbb C^d$. There exists a unique solution $f=E_\kappa^W(\cdot,y)$ of 
\[T^{\mathcal R}_\xi f=\langle y,\xi\rangle f \quad \forall \xi \in \mathbb R^d\]
which is real-analytic on $\mathbb R^d$ and satisfies $f(0)=1$. Moreover the Dunkl kernel $E_\kappa^W$ extends to a holomorphic function on $\mathbb C^d\times \mathbb C^d$.
\end{theo}
In fact, the existence of a solution has already been  proved by Dunkl (\cite{Dunkl91}). Indeed, he noticed the existence of an intertwining operator $V_\kappa^W$ which satisfies
\[
V^{W}_\kappa(\mathcal P^d_n)=\mathcal P^d_n, \quad \ {V^{W}_\kappa}_{|_{\mathcal P^d_0}}=\mathrm{Id}_{|_{\mathcal  P^d_0}},\quad \ T^{\mathcal R}_\xi V^{W}_\kappa= V^{W}_\kappa\partial_\xi \ \, \, \forall \xi \in \mathbb R^d,
\]
where $\mathcal P_n$ denotes the space of homogeneous polynomials of degree $n$ in $d$ variables. Since the exponential function is a solution of \eqref{spectral} when $\kappa=0$ (that is,  when $T^{\mathcal R}_\xi=\partial_\xi$), he naturally set $E_\kappa(\cdot,y)=V_\kappa^W(\mathrm{e}^{\langle\cdot,y\rangle})$. Unfortunately, the Dunkl kernel is explicitly known only in some special cases; when the root system is of $A_2$-type (\cite{Dunkl95}), of $B_2$-type (\cite{Dunkl07}) and when the reflection group is $\mathbb Z_2^d$ (\cite{DunklHankel,XuEntrelacement}). Nevertheless we know that this kernel has many properties in common with the classical exponential to which  it reduces when $\kappa=0$. For significant results on this kernel and the intertwining operator, the reader may especially consult  \cite{deJeuInv,Dunkl91,Opdam,RoslerHyp,RoslerPositivity,RoslerTrans}. The Dunkl kernel is of particular interest as it gives rise to an integral transform which is taken with respect to a weighted Lebesgue measure invariant under the action of $W$ and which generalizes the Euclidean Fourier transform.
\newline \indent
More precisely, let us introduce the measure $\mathrm{d}\mu_\kappa^W(x)=h_\kappa^2(x)\,\mathrm{d}x$ where the weight given by
\[h_\kappa^2(x)=\prod_{\alpha \in \mathcal R_+}|\langle x,\alpha\rangle|^{2\kappa(\alpha)}\]
is homogeneous of degree $2\gamma$ with
\[\gamma=\sum_{\alpha \in \mathcal R_+}\kappa(\alpha).\] 
Then for every $f \in L^1(\mu_\kappa^W)$, the Dunkl transform of $f$, denoted by $\mathcal F_\kappa^W(f)$, is defined by
\[
\mathcal F_\kappa^W(f)(x)=c_\kappa^W\int_{\mathbb R^d}E^W_\kappa(-ix,y)f(y)\,\mathrm{d}\mu_\kappa^W(y), \quad x \in \mathbb R^d, 
\]
where $c_\kappa^W$ is the Mehta-type constant
\[
c_\kappa^W=\biggl(\int_{\mathbb R^d}\mathrm{e}^{-\frac{\|x\|^2}{2}}\,\mathrm{d}\mu^W_{\kappa}(x)\biggr)^{-1}.
\]
Let us point out that the Dunkl transform coincides with the Euclidean Fourier transform when $\kappa=0$ and that it is more or less a Hankel transform when $d=1$. The two main properties of the Dunkl transform are given in the following  theorem (\cite{deJeuInv,DunklHankel}).
\begin{theo} \label{InversionPlancherel}
\mbox{}
\begin{enumerate}
\item
\textbf{Inversion formula}\ \  Let $f \in L^1(\mu_\kappa^W)$. If $\mathcal F_\kappa^W(f)$ is in $L^1(\mu_\kappa^W)$, then we have the following inversion formula:
\[
 f(x)=c_\kappa^W\int_{\mathbb R^d}E^W_\kappa(ix,y)\mathcal F_\kappa^W(f)(y)\,\mathrm{d}\mu_\kappa^W(y).
\]
\item
\textbf{Plancherel theorem}\ \  The Dunkl transform has a unique extension to an isometric isomorphism of $L^2(\mu_\kappa^W)$.
\end{enumerate}
\end{theo}

The Dunkl transform shares many other properties with the Fourier transform. Therefore, it is natural to associate a generalized translation operator with this transform.
\newline \indent There are many ways to define the Dunkl translation but we use the definition which most underlines the analogy with the Fourier transform. It is the definition given in \cite{ThangaXuMaximal} with a different convention. Let $x \in \mathbb R^d$. The Dunkl translation  $f\mapsto\tau_x^{W}f$ is defined on $L^2(\mu_\kappa^W)$ by the equation
\[
\mathcal F_\kappa^W(\tau_x^{W}f)(y)=E_\kappa^W(ix,y)\mathcal F_\kappa^W(f)(y),\quad y \in \mathbb R^d.
\]
It is useful to have a set of functions for which the above equality holds pointwise. One can take the set
\[
\mathcal A_\kappa^W(\mathbb R^d)=\bigl\{f \in L^1(\mu_\kappa^W): \mathcal F_\kappa^W(f) \in L^1(\mu_\kappa^W)\bigr\},
\]
which is a subset of $L^2(\mu_\kappa^W)$ (since it is contained in the intersection of $L^1(\mu_\kappa^W)$ and $L^\infty$). For $f \in \mathcal A_\kappa^W(\mathbb R^d)$, the inversion formula allows us to write
\[
 \tau_x^{W}f(y)=c_\kappa^W\int_{\mathbb R^d}E_\kappa^W(ix,z)E_\kappa^W(iy,z)\mathcal F_\kappa^W(f)(z)\,\mathrm{d}\mu^W_\kappa(z).
\]
\indent  In Fourier analysis, the translation operator $f\mapsto f(\cdot+x)$ (to which the Dunkl translation reduces when $\kappa=0$) is positive and $L^p$-bounded. In the Dunkl setting, $\tau_x^W$ is not a positive operator (\cite{RoslerHyp,ThangaXuMaximal}) and the $L^p(\mu_\kappa^W)$-boundedness is still a challenging problem, apart from the trivial case where $p=2$ (thanks to the Plancherel theorem and the fact that $|E_\kappa^W(ix,y)|\leqslant1$). The most general result we have is given in the following theorem (\cite{RoslerTrans,ThangaXuMaximal}), where we denote by $L^p_{\mathrm{rad}}(\mu_\kappa^W)$ the space of radial functions in $L^p(\mu_\kappa^W)$.

\begin{theo} \label{TransRad}
\mbox{}
\begin{enumerate}
\item For every $p$ satisfying $1\leqslant p\leqslant 2$ and for every $x \in \mathbb R^d$, the Dunkl translation $\tau_x^{W}: L^p_{\mathrm{rad}}(\mu_\kappa^W)\to L^p(\mu_\kappa^W)$ is a bounded operator.
\item Let $f \in L^1(\mu_\kappa^W)$ be a bounded, radial and positive function. Then $\tau_x^{W}f\geqslant0$ for every $x \in \mathbb R^d$.
\end{enumerate}
\end{theo}

The last result we mention about the Dunkl translation is the following.

\begin{theo} \label{ega}
 Let $f \in L^1_{\mathrm{rad}}(\mu_\kappa^W)$. Then, for every $x \in \mathbb R^d$,  
\[
 \int_{\mathbb R^d}\tau_x^{W}f(y)\,\mathrm{d}\mu_\kappa^W(y)=\int_{\mathbb R^d}f(y)\,\mathrm{d}\mu_\kappa^W(y).
\]
\end{theo}

Another important tool in the Dunkl analysis is the Dunkl-type heat semi-group which has been mainly studied by Rösler (\cite{RoslerHermite,RoslerQuantum}).
We are searching for solutions $u\in \mathcal C^2(\mathbb R^d\times]0,+\infty[)\cap\mathcal C_b(\mathbb R^d\times[0,+\infty[)$ of the following Cauchy problem for the generalized heat equation:
\[
(CH)_\kappa:\begin{cases}\Delta_\kappa^Wu(x,t)&=\partial_tu(x,t) \quad \forall (x,t) \in \mathbb R^d\times]0,+\infty[, \\ u(\cdot,0)&=f \end{cases}
\]
with initial data $f$ in the Schwartz space $\mathcal S(\mathbb R^d)$ and where $\Delta_\kappa^W=\sum_{j=1}^d(T_{e_j}^{\mathcal R})^2$ is the Dunkl Laplacian.
It is easily noticed that a solution of $\Delta_\kappa^W u(x,t)=\partial_tu(x,t)$ is given on $\mathbb R^d\times]0,+\infty[$ by the generalized Gaussian which is defined for every $x \in \mathbb R^d$ by
\[
 q_t^{W}(x)=\frac{c_\kappa^W}{(2t)^{\frac{d}{2}+\gamma}}\mathrm{e}^{-\frac{\|x\|^2}{4t}}
\]
and which has the following two properties.

\begin{pr} \label{gau1}
\mbox{}
\begin{enumerate}
\item For every $t>0$, 
\[
 \int_{\mathbb R^d}q_t^{W}(x)\,\mathrm{d}\mu_\kappa^W(x)=1.
\]
\item For every $t>0$ and every $x \in \mathbb R^d$,
\[
 \mathcal F_\kappa^W(q_t^{W})(x)=c_\kappa^W\mathrm{e}^{-t\|x\|^2}.
\]
\end{enumerate}
\end{pr}

The Dunkl-type heat kernel $Q_\kappa^W$ is defined by taking the Dunkl translation of $q_t^{W}$, that is, according to \cite{RoslerHermite},
\begin{align*}
Q_\kappa^W(x,y,t)&=\tau_x^Wq_t^W(-y)\\
&=\frac{c_\kappa^W}{(2t)^{\frac{d}{2}+\gamma}}\mathrm{e}^{-\frac{(\|x\|^2+\|y\|^2)}{4t}}E_\kappa^W\Bigl(\frac{x}{\sqrt{2t}},\frac{y}{\sqrt{2t}}\Bigr),\quad x,y \in \mathbb R^d, t>0. 
\end{align*}

This positive kernel (\cite{RoslerHermite}) allows us to define a generalized heat operator (or Dunkl-type heat operator). More precisely for every $f \in L^p(\mu_\kappa^W)$, with $1\leqslant p \leqslant+\infty$, and for every $t\geqslant0$, we set
\[H_t^{W}f=\begin{cases}\int_{\mathbb R^d}f(y)Q_\kappa^W(\cdot,y,t)\,\mathrm{d}\mu_\kappa^W(y) &\text{if}\ t>0,
\\ f &\text{if}\ t=0.
 \end{cases}\]

The following fundamental result about this operator is due to Rösler (see \cite{RoslerHermite} and \cite{RoslerQuantum}).

\begin{theo} \label{rosch}
For every $p$ satisfying $1\leqslant p \leqslant+\infty$, the family $\{H_t^{W}\}_{t\geqslant0}$ is a positive and contraction semigroup on $L^p(\mu_\kappa^W)$. Moreover, for every $f \in \mathcal S(\mathbb R^d)$, the function $u$ given for every $(x,t) \in \mathbb R^d\times[0,+\infty[$ by
\[ u(x,t)=H_t^{W}f(x) \]
belongs to $\mathcal C^2(\mathbb R^d\times]0,+\infty[)\cap\mathcal C_b(\mathbb R^d\times[0,+\infty[)$ and is a solution of the Cauchy problem $(CH)_\kappa$. 
\end{theo}

We can easily improve the previous theorem.

\begin{theo} \label{diffusion}
For every $p$ satisfying $1\leqslant p \leqslant+\infty$, the family $\{H_t^{W}\}_{t\geqslant0}$ is a symmetric diffusion semigroup on $L^p(\mu_\kappa^W)$, that is,  a semigroup which satisfies:
\begin{enumerate}
\item $H_t^{W}$ is a contraction on $L^p(\mu_\kappa^W)$;
\item $H_t^{W}$ is symmetric, that is,  self-adjoint on $L^2(\mu_\kappa^W)$;
\item $H_t^{W}$ is positive;
\item $H_t^{W}(1)=1$.
\end{enumerate}
\end{theo}

The reader is referred to the book of Stein \cite{SteinTopics} for a detailed study of this kind of semigroup.

\section{Behavior of the Dunkl maximal operator in the scalar case}

 In this section we give the proof of Theorem \ref{firstresult} for a general reflection group. For convenience, we prove each point of the theorem separately. Thus, we first establish the following result.

\begin{theo} \label{ConstanteFaible}
There exists a numerical constant $C$ such that for every $f \in L^1(\mu_\kappa^W)$ and every $\lambda>0$, 
\[
\mu_\kappa^W\Bigl(\Bigl\{x \in \mathbb{R}^d: M_\kappa^Wf(x)>\lambda\Bigr\}\Bigr)\leqslant C\frac{d+2\gamma}{\lambda}\|f\|_{W,1}.
\]
\end{theo}

To prove this theorem, we need two lemmas. The first one is  basic calculus. Before stating it, we introduce some notation.
 
\begin{nota}
We define
\[
a(S^{d-1})=\int_{S^{d-1}}h^2_{\kappa}(x)\,\mathrm{d}\omega(x)
\]
where $\omega$ is the usual Lebesgue measure on $S^{d-1}$. We also write
\[
 {q_t^{W}}_{|_{S^{d-1}}}=\frac{c_\kappa^W}{(2t)^{\frac{d}{2}+\gamma}}\mathrm{e}^{-\frac{1}{4t}}.
\]
\end{nota}

With this notation, we can formulate the lemma.

\begin{lem} \label{CF1}
 We have the following equalities:
\begin{align*}
&\mu_\kappa^W(B_1)=\frac{a(S^{d-1})}{d+2\gamma}, \\
 &(c_\kappa^W)^{-1}=2^{\frac{d}{2}+\gamma-1}\Gamma\Bigl(\frac{d}{2}+\gamma\Bigr)a(S^{d-1}), \\
&\int_0^{+\infty}{q_t^{W}}_{|_{S^{d-1}}}\,\mathrm{d}t=\frac{c_\kappa^W2^{\frac{d}{2}+\gamma}}{4}\Gamma\Bigl(\frac{d}{2}+\gamma-1\Bigr).
\end{align*}
Moreover if  $d+2\gamma\geqslant8$, then 
\[
 \int_{\frac{1}{d+2\gamma}}^{+\infty}{q_t^{W}}_{|_{S^{d-1}}}\,\mathrm{d}t\leqslant\frac{c_\kappa^W2^{\frac{d}{2}+\gamma}}{4}\Bigl(\frac{d+2\gamma}{4}\Bigr)^{\frac{d}{2}+\gamma-1}\mathrm{e}^{-\frac{d+2\gamma}{4}}.
\]
\end{lem}

\begin{proof} 
The equalities are easy to prove by passing to polar coordinates and by substitution. Therefore we only prove the inequality. By definition we have
\[
\int_{\frac{1}{d+2\gamma}}^{+\infty}{q_t^{W}}_{|_{S^{d-1}}}\,\mathrm{d}t=c_\kappa^W\int_{\frac{1}{d+2\gamma}}^{+\infty}\frac{1}{(2t)^{\frac{d}{2}+\gamma}}\mathrm{e}^{-\frac{1}{4t}}\,\mathrm{d}t,
\]
which leads after a change of variables to
\[
 \int_{\frac{1}{d+2\gamma}}^{+\infty}{q_t^{W}}_{|_{S^{d-1}}}\,\mathrm{d}t=\frac{c_\kappa^W2^{\frac{d}{2}+\gamma}}{4}\int_{0}^{\frac{d+2\gamma}{4}}t^{\frac{d}{2}+\gamma-2}\mathrm{e}^{-t}\,\mathrm{d}t.
\]
Since for $d+2\gamma\geqslant8$ and $t \in [0,\frac{d+2\gamma}{4}]$ we have
\[
t^{\frac{d}{2}+\gamma-2}\mathrm{e}^{-t}\leqslant \Bigl(\frac{d+2\gamma}{4}\Bigr)^{\frac{d}{2}+\gamma-2}\mathrm{e}^{-\frac{d+2\gamma}{4}},
\]
we  obtain
\[
 \int_{0}^{\frac{d+2\gamma}{4}}t^{\frac{d}{2}+\gamma-2}\mathrm{e}^{-t}\,\mathrm{d}t\leqslant \Bigl(\frac{d+2\gamma}{4}\Bigr)^{\frac{d}{2}+\gamma-1}\mathrm{e}^{-\frac{d+2\gamma}{4}},
\]
and the inequality is proved.
\end{proof}

The second lemma allows us to reduce the inequality of Theorem \ref{ConstanteFaible} to a more convenient one.

\begin{lem} \label{CF2}
If there exists $t_0>0$ such that
\[
 \frac{1}{\mu_\kappa^W(B_1)}\leqslant\frac{C(d,\kappa)}{t_0}\int_0^{t_0}{q_t^{W}}_{|_{S^{d-1}}}\,\mathrm{d}t,
\]
then  for every $f \in L^1(\mu_\kappa^W)$ and every $\lambda>0$,
\[
 \mu_\kappa^W\Bigl(\Bigl\{x \in \mathbb{R}^d: M_\kappa^Wf(x)>\lambda\Bigr\}\Bigr)\leqslant 4\frac{C(d,\kappa)}{\lambda}\|f\|_{W,1},
\]
where $C(d,\kappa)$ is the same positive constant in both hypothesis and conclusion of the lemma.
\end{lem}

\begin{proof} We can assume that $f$ is nonnegative. If there exists $t_0>0$ such that
\begin{equation*} 
\frac{1}{\mu_\kappa^W(B_1)}\leqslant\frac{C(d,\kappa)}{t_0}\int_0^{t_0}{q_t^{W}}_{|_{S^{d-1}}}\,\mathrm{d}t,
\end{equation*}
we can deduce that for $n$ large enough,
\begin{equation} \label{largeenough}
\frac{1}{\mu_\kappa^W(B_1)}\leqslant\frac{2C(d,\kappa)}{t_0}\int_{\frac{1}{n}}^{t_0}{q_t^{W}}_{|_{S^{d-1}}}\,\mathrm{d}t.
\end{equation}
Then we claim that  for every  $y \in \mathbb{R}^d$,
\begin{equation} \label{tailleint1}
\frac{1}{\mu_\kappa^W(B_1)}\chi_{{}_{B_1}}(y)\leqslant\frac{2C(d,\kappa)}{t_0}\int_{\frac{1}{n}}^{t_0}q_t^{W}(y)\,\mathrm{d}t.
\end{equation}
Indeed if $\|y\|>1$, there is nothing to do since $\chi_{{}_{B_1}}(y)=0$. If $\|y\|\leqslant1$, it is enough to use \eqref{largeenough} and the fact that ${q_t^{W}}_{|_{S^{d-1}}}\leqslant q_t^{W}(y)$.
\newline 
As a result, we can write, for every $r>0$ and every $y \in \mathbb{R}^d$,
\[
 \frac{1}{\mu_\kappa^W(B_r)}\chi_{{}_{B_r}}(y)=\frac{1}{r^{d+2\gamma}\mu_\kappa^W(B_1)}\chi_{{}_{B_1}}\Bigl(\frac{y}{r}\Bigr)\leqslant\frac{2C(d,\kappa)}{r^{d+2\gamma}t_0}\int_{\frac{1}{n}}^{t_0}q_t^{W}\Bigl(\frac{y}{r}\Bigr)\,\mathrm{d}t,
\]
which leads after a change of variables to
\[
\frac{1}{\mu_\kappa^W(B_r)}\chi_{{}_{B_r}}(y)\leqslant \frac{2C(d,\kappa)}{r^2t_0}\int_{\frac{r^2}{n}}^{r^2t_0}q_t^{W}(y)\,\mathrm{d}t.
\]
Let $x \in \mathbb{R}^d$. By Theorem \ref{TransRad}$(2)$,
\[
\frac{1}{\mu_\kappa^W(B_r)}\tau_x^{W}(\chi_{{}_{B_r}})(-y)\leqslant \frac{2C(d,\kappa)}{r^2t_0}\tau_x^{W}\biggl(\int_{\frac{r^2}{n}}^{r^2t_0}q_t^{W}(\cdot)\,\mathrm{d}t\biggr)(-y).
\]
Let us temporarily assume that
\begin{equation} \label{tempconst}
\tau_x^{W}\biggl(\int_{\frac{r^2}{n}}^{r^2t_0}q_t^{W}(\cdot)\,\mathrm{d}t\biggr)(-y)=\int_{\frac{r^2}{n}}^{r^2t_0}\tau_x^{W}(q_t^{W})(-y)\,\mathrm{d}t.
\end{equation}
This implies
\[
\frac{1}{\mu_\kappa^W(B_r)}\tau_x^{W}(\chi_{{}_{B_r}})(-y)\leqslant \frac{2C(d,\kappa)}{r^2t_0}\int_{\frac{r^2}{n}}^{r^2t_0}\tau_x^{W}(q_t^{W})(-y)\,\mathrm{d}t.
\]
Multiplying both sides by $f(y)$ and integrating over $\mathbb{R}^d$  we obtain
\[
\frac{1}{\mu_\kappa^W(B_r)}\int_{\mathbb{R}^d}f(y)\tau_x^{W}(\chi_{{}_{B_r}})(-y)\,\mathrm{d}\mu_\kappa^W(y)\leqslant \frac{2C(d,\kappa)}{r^2t_0}\int_{\frac{r^2}{n}}^{r^2t_0}H_t^{W}f(x)\,\mathrm{d}t,
\]
and so
\[
\frac{1}{\mu_\kappa^W(B_r)}\int_{\mathbb{R}^d}f(y)\tau_x^{W}(\chi_{{}_{B_r}})(-y)\,\mathrm{d}\mu_\kappa^W(y)\leqslant \frac{2C(d,\kappa)}{r^2t_0}\int_{0}^{r^2t_0}H_t^{W}f(x)\,\mathrm{d}t,
\]
This easily yields
\[
\frac{1}{\mu_\kappa^W(B_r)}\int_{\mathbb{R}^d}f(y)\tau_x^{W}(\chi_{{}_{B_r}})(-y)\,\mathrm{d}\mu_\kappa^W(y)\leqslant 2C(d,\kappa)\sup_{s>0}\biggl(\frac{1}{s}\int_0^{s}H_t^{W}f(x)\,\mathrm{d}t\biggr),
\]
from which we can deduce that
\[
M_\kappa^Wf(x)\leqslant 2C(d,\kappa)\sup_{s>0}\biggl(\frac{1}{s}\int_0^{s}H_t^{W}f(x)\,\mathrm{d}t\biggr). 
\]
Since $\{H_t^{W}\}_{t\geqslant0}$ is a semigroup which satisfies the contraction property on $L^p(\mu_\kappa^W)$ (Theorem \ref{rosch}), the first point of the Hopf-Dunford-Schwartz ergodic theorem implies the desired conclusion and we are left with the task of establishing \eqref{tempconst}.

 Let $n$ be an integer. Due to Proposition \ref{gau1}, it is easily seen that the radial function $\int_{\frac{r^2}{n}}^{r^2t_0}q_t^{W}(\cdot)\,\mathrm{d}t$ is in $L^1(\mu_\kappa^W)$ so  Theorem \ref{TransRad} leads  to
\[\tau_x^{W}\biggl(\int_{\frac{r^2}{n}}^{r^2t_0}q_t^{W}(\cdot)\,\mathrm{d}t\biggr) \in L^1(\mu_\kappa^W).\]
 Since  on one hand
\[\mathcal F_\kappa^W\biggl(\int_{\frac{r^2}{n}}^{r^2t_0}q_t^{W}(\cdot)\,\mathrm{d}t\biggr)=\int_{\frac{r^2}{n}}^{r^2t_0}c_\kappa^W\mathrm{e}^{-t\|\cdot\|^2}\,\mathrm{d}t \in L^1(\mu_\kappa^W),\]
and on the other hand
\[\mathcal F_\kappa^W\biggl(\tau_x^{W}\biggl(\int_{\frac{r^2}{n}}^{r^2t_0}q_t^{W}(\cdot)\,\mathrm{d}t\biggr)\biggr)=E_\kappa^W(ix,\cdot)\,\mathcal F_\kappa^W\biggl(\int_{\frac{r^2}{n}}^{r^2t_0}q_t^{W}(\cdot)\,\mathrm{d}t\biggr),\]
we find
\[\mathcal F_\kappa^W\biggl(\tau_x^{W}\biggl(\int_{\frac{r^2}{n}}^{r^2t_0}q_t^{W}(\cdot)\,\mathrm{d}t\biggr)\biggr) \in L^1(\mu_\kappa^W).\]
Consequently, $\tau_x^{W}\bigl(\int_{\frac{r^2}{n}}^{r^2t_0}q_t^{W}(\cdot)\,\mathrm{d}t\bigr) \in \mathcal A_\kappa^W(\mathbb{R}^d)$ and we can use the inversion formula to obtain
\[\tau_x^{W}\biggl(\int_{\frac{r^2}{n}}^{r^2t_0}q_t^{W}(\cdot)\,\mathrm{d}t\biggr)(y)=c_\kappa^W\int_{\mathbb{R}^d}E_\kappa^W(ix,z)E_\kappa^W(iy,z)\int_{\frac{r^2}{n}}^{r^2t_0}c_\kappa^W\mathrm{e}^{-t\|z\|^2}\,\mathrm{d}t\,\mathrm{d}\mu_\kappa^W(z), \]
that is,
\begin{multline*} \tau_x^{W}\biggl(\int_{\frac{r^2}{n}}^{r^2t_0}q_t^{W}(\cdot)\,\mathrm{d}t\biggr)(y)\\
=\int_{\frac{r^2}{n}}^{r^2t_0}\biggl(c_\kappa^W\int_{\mathbb{R}^d}E_\kappa^W(ix,z)E_\kappa^W(iy,z)c_\kappa^W\mathrm{e}^{-t\|z\|^2}\,\mathrm{d}\mu_\kappa^W(z)\biggr)\,\mathrm{d}t.
\end{multline*}
The  inversion formula leads  to
\[\tau_x^{W}\biggl(\int_{\frac{r^2}{n}}^{r^2t_0}q_t^{W}(\cdot)\,\mathrm{d}t\biggr)(y)=\int_{\frac{r^2}{n}}^{r^2t_0}\tau_x^{W}(q_t^{W})(y)\,\mathrm{d}t,\]
and \eqref{tempconst} is established. The lemma is proved.\end{proof}

\begin{rem}
Of course, we can take the limit of \eqref{tempconst}  to prove 
\begin{equation} \label{boff}
\tau_x^{W}\biggl(\int_{0}^{r^2t_0}q_t^{W}(\cdot)\,\mathrm{d}t\biggr)(y)=\int_{0}^{r^2t_0}\tau_x^{W}(q_t^{W})(y)\,\mathrm{d}t.
\end{equation}
 On one hand, $\int_{\frac{r^2}{n}}^{r^2t_0}q_t^{W}(\cdot)\,\mathrm{d}t \to \int_{0}^{r^2t_0}q_t^{W}(\cdot)\,\mathrm{d}t$ in $L^1(\mu_\kappa^W)$ as $n$ goes to infinity. Therefore, by Theorem \ref{TransRad}$(1)$,
\[\tau_x^{W}\biggl(\int_{\frac{r^2}{n}}^{r^2t_0}q_t^{W}(\cdot)\,\mathrm{d}t\biggr) \to \tau_x^{W}\biggl(\int_{0}^{r^2t_0}q_t^{W}(\cdot)\,\mathrm{d}t\biggr)\]
in $L^1(\mu_\kappa^W)$ as $n$ goes to infinity. \newline

On the other hand
\[0\leqslant\int_{0}^{r^2t_0}\tau_x^{W}(q_t^{W})(y)\,\mathrm{d}t-\int_{\frac{r^2}{n}}^{r^2t_0}\tau_x^{W}(q_t^{W})(y)\,\mathrm{d}t=\int^{\frac{r^2}{n}}_0\tau_x^{W}(q_t^{W})(y)\,\mathrm{d}t.\]
But 
\[\int_{\mathbb{R}^d}\int^{\frac{r^2}{n}}_0\tau_x^{W}(q_t^{W})(y)\,\mathrm{d}t\,\mathrm{d}\mu_\kappa^W(y)=\int^{\frac{r^2}{n}}_0\int_{\mathbb{R}^d}\tau_x^{W}(q_t^{W})(y)\,\mathrm{d}\mu_\kappa^W(y)\,\mathrm{d}t=\frac{r^2}{n}\]
since  thanks to Theorem \ref{ega},
\[ \int_{\mathbb{R}^d}\tau_x^{W}(q_t^{W})(y)\,\mathrm{d}\mu_\kappa^W(y)=\int_{\mathbb{R}^d}q_t^{W}(y)\,\mathrm{d}\mu_\kappa^W(y)=1.\]
Consequently
\[ \int_{\frac{r^2}{n}}^{r^2t_0}\tau_x^{W}(q_t^{W})(y)\,\mathrm{d}t\to\int_{0}^{r^2t_0}\tau_x^{W}(q_t^{W})(y)\,\mathrm{d}t\]
in $L^1(\mu_\kappa^W)$ as $n$ goes to infinity. The equality \eqref{boff} is thus true. 
\end{rem}

We are now in a position to prove Theorem \ref{ConstanteFaible}.

\begin{proof} Due to Lemma \ref{CF2}, it is enough to find $t_0>0$ such that 
\begin{equation} \label{CFint1}
 \frac{1}{\mu_\kappa^W(B_1)}\leqslant C\frac{d+2\gamma}{t_0}\int_0^{t_0}{q_t^{W}}_{|_{S^{d-1}}}\,\mathrm{d}t,
\end{equation}
where $C$ is a numerical constant.
Set $t_0=\tfrac{1}{d+2\gamma}$. The inequality of Lemma \ref{CF1} asserts that for $d+2\gamma\geqslant8$,
\[
\int_{\frac{1}{d+2\gamma}}^{+\infty}{q_t^{W}}_{|_{S^{d-1}}}\,\mathrm{d}t\leqslant\frac{c_\kappa^W2^{\frac{d}{2}+\gamma}}{4}\Bigl(\frac{d+2\gamma}{4}\Bigr)^{\frac{d}{2}+\gamma-1}\mathrm{e}^{-\frac{d+2\gamma}{4}}.
\]
 On one hand, Stirling's formula implies that
\[
\Bigl(\frac{n}{4}\Bigr)^{\frac{n}{2}-1}\mathrm{e}^{-\frac{n}{4}}=\underset{n \to +\infty}{\mathrm{o}}\Bigl(\Gamma\Bigl(\frac{n}{2}-1\Bigr)\Bigr),
\]
 and, on the other hand, the third equality of Lemma \ref{CF1} yields 
\[
 \int_0^{+\infty}{q_t^{W}}_{|_{S^{d-1}}}\,\mathrm{d}t=\frac{c_\kappa^W2^{\frac{d}{2}+\gamma}}{4}\Gamma\Bigl(\frac{d}{2}+\gamma-1\Bigr).
\]
Therefore, we conclude that there exists a numerical constant $C$ such that 
\begin{equation} \label{nvl}
c_\kappa^W2^{\frac{d}{2}+\gamma}\Gamma\Bigl(\frac{d}{2}+\gamma-1\Bigr)\leqslant  C\int_0^{\frac{1}{d+2\gamma}}{q_t^{W}}_{|_{S^{d-1}}}\,\mathrm{d}t.
\end{equation}
By the first two equalities of Lemma \ref{CF1} we can write
\[
 c_\kappa^W2^{\frac{d}{2}+\gamma}\Gamma\Bigl(\frac{d}{2}+\gamma-1\Bigr)=\frac{2\,\Gamma\bigl(\frac{d}{2}+\gamma-1\bigr)}{\mu_\kappa^W(B_1)(d+2\gamma)\Gamma\bigl(\frac{d}{2}+\gamma\bigr)},
\]
and by inserting this  in \eqref{nvl} we are led to 
\[
 \frac{1}{\mu_\kappa^W(B_1)}\leqslant C\frac{(d+2\gamma)\,\Gamma\bigl(\frac{d}{2}+\gamma\bigr)}{\Gamma\bigl(\frac{d}{2}+\gamma-1\bigr)}\int_0^{\frac{1}{d+2\gamma}}{q_t^{W}}_{|_{S^{d-1}}}\,\mathrm{d}t,
\]
which finally implies that 
\[
 \frac{1}{\mu_\kappa^W(B_1)}\leqslant C(d+2\gamma)^2\int_0^{\frac{1}{d+2\gamma}}{q_t^{W}}_{|_{S^{d-1}}}\,\mathrm{d}t.
\]
This proves \eqref{CFint1} and  establishes the theorem.
\end{proof}

We now turn to Theorem \ref{firstresult}$(2)$ that we recall below.

\begin{theo} \label{ConstanteForte}
There exists a numerical constant $C$ such that for every $f \in L^p(\mu_\kappa^W)$ with $1<p<+\infty$, 
\[
\bigl\|M_\kappa^Wf\bigr\|_{W,p}\leqslant C\frac{p}{p-1}\sqrt{d+2\gamma}\,\|f\|_{W,p}.
\]
\end{theo}

\begin{rem}
This result  is better that what one would obtain by using Theorem \ref{ConstanteFaible}, the $L^\infty$ case and the Marcinkiewicz interpolation theorem.
\end{rem}

In order to prove Theorem \ref{ConstanteForte}, we need the following lemma which reduces the inequality of that theorem to a more convenient one.

\begin{lem} \label{CFO}
If there exists $t_0>0$ such that
\[
 \frac{1}{\mu_\kappa^W(B_1)}\leqslant C(d,\kappa){q_{t_0}^{W}}_{|_{S^{d-1}}},
\]
then there exists a numerical constant $C$ such that for every $p$ satisfying $1<p<+\infty$ and every $f \in L^p(\mu_\kappa^W)$,
\[
 \bigl\|M_\kappa^Wf\bigr\|_{W,p}\leqslant C\frac{p}{p-1}C(d,\kappa)\,\|f\|_{W,p},
\]
where $C(d,\kappa)$ is the same positive constant in both  hypothesis and  conclusion of the lemma.
\end{lem}

\begin{proof}
The proof is quite similar to the one of Lemma \ref{CF2}. One can assume that $f$ is nonnegative. If there exists $t_0>0$ such that 
\begin{equation} \label{tailleint2}
 \frac{1}{\mu_\kappa^W(B_1)}\leqslant C(d,\kappa){q_{t_0}^{W}}_{|_{S^{d-1}}},
\end{equation}
then we can deduce that for every $y \in \mathbb{R}^d$,
\begin{equation} \label{tailleint3}
 \frac{1}{\mu_\kappa^W(B_1)}\chi_{{}_{B_1}}(y)\leqslant C(d,\kappa)q_{t_0}^{W}(y).
\end{equation}
Indeed, if $\|y\|>1$, it is obvious since $\chi_{{}_{B_1}}(y)=0$. If $\|y\|\leqslant1$, it is enough to use \eqref{tailleint2} and the fact that ${q_{t_0}^{W}}_{|_{S^{d-1}}}\leqslant q_{t_0}^{W}(y)$.

Consequently, for every $r>0$, we can write
\begin{align*}
 \frac{1}{\mu_\kappa^W(B_r)}\chi_{{}_{B_r}}(y)&=\frac{1}{r^{d+2\gamma}\mu_\kappa^W(B_1)}\chi_{{}_{B_1}}\Bigl(\frac{y}{r}\Bigr)\\
 &\leqslant \frac{C(d,\kappa)}{r^{d+2\gamma}}q_{t_0}^{W}\Bigl(\frac{y}{r}\Bigr)=C(d,\kappa)q_{r^2t_0}^{W}(y).
\end{align*}
Let $x \in \mathbb R^d$. Due to Theorem \ref{TransRad}$(2)$ we can assert that
\[
\frac{1}{\mu_\kappa^W(B_r)}\tau_x^{W}(\chi_{{}_{B_r}})(-y)\leqslant C(d,\kappa)\tau_x^{W}(q^{W,\kappa}_{r^2t_0})(-y).
\]
Multiplying both sides by $f(y)$ and integrating over $\mathbb R^d$ we obtain
\[
\frac{1}{\mu_\kappa^W(B_r)}\int_{\mathbb R^d}f(y)\tau_x^{W}(\chi_{{}_{B_r}})(-y)\,\mathrm{d}\mu_\kappa^W(y)\leqslant C(d,\kappa)H^{W,\kappa}_{r^2t_0}f(x),
\]
from which we deduce
\begin{equation} \label{CFOint1}
M_\kappa^Wf(x)\leqslant C(d,\kappa)\sup_{t>0}H_t^{W}f(x).
\end{equation}
But $\{H^{W}_t\}_{t\geqslant 0}$ is a symmetric diffusion semigroup  on $L^p(\mu_\kappa^W)$ (Theorem \ref{diffusion}). Therefore, by a result due to Stein (see \cite[chapter 4]{SteinTopics}),  for every $p$ satisfying $1<p<+\infty$,
\[
\bigl\|\sup_{t>0}H_t^{W}f\bigr\|_{W,p}\leqslant C\frac{p}{p-1}\|f\|_{W,p} ,
\]
where $C$ is a numerical constant. By using this inequality in \eqref{CFOint1} we obtain the desired result. 
\end{proof}

We can now turn to the proof of Theorem \ref{ConstanteForte}.

\begin{proof}
Thanks to the previous lemma, it is enough to find $t_0>0$ such that 
\[
\frac{1}{\mu_\kappa^W(B_1)}\leqslant C\sqrt{d+2\gamma}\,{q_{t_0}^{W}}_{|_{S^{d-1}}}
\]
or, equivalently, such that
\[
\frac{1}{\mu_\kappa^W(B_1)}\leqslant Cc_\kappa^W\sqrt{d+2\gamma}\Bigl(\frac{1}{2t_0}\Bigr)^{\frac{d}{2}+\gamma}\mathrm{e}^{-\frac{1}{4t_0}}.
\]
On one hand, the first two equalities of Lemma \ref{CF1} allow us to write
\[
\frac{1}{\mu_\kappa^W(B_1)}=c_\kappa^W(d+2\gamma)2^{\frac{d}{2}+\gamma-1}\Gamma\Bigl(\frac{d}{2}+\gamma\Bigr), 
\]
and on the other hand, Stirling's formula gives us 
\[
 2^\frac{n}{2}\Gamma\Bigl(\frac{n}{2}\Bigr)=\underset{n \to +\infty}{\mathrm{O}}\bigl(n^{\frac{n-1}{2}}\mathrm{e}^{-\frac{n}{2}}\bigr).
\]
We finally obtain the desired result by choosing $t_0=\tfrac{1}{2d+4\gamma}$.
\end{proof}

\section{Exponential integrability in the vector-valued case}

This section is devoted to the proof of Theorem \ref{expint} in the case where the reflection group is $\mathbb Z_2^d$ (that is, the root system is $\mathcal R=\{\pm e_j:1\leqslant j\leqslant d\}$). Let us recall some facts related to Dunkl's analysis associated with this particular reflection group.

In this case, an explicit formula of the intertwining operator $V_\kappa^{\mathbb Z_2^d}$ is known (see \cite{XuEntrelacement}) and there is an explicit formula for the Dunkl translation. In the one-dimensional case, the following formula has been proved by Rösler in the setting of signed hypergroups (see \cite{RoslerHyp}):
\begin{multline} \label{extau}
\tau^{\mathbb Z_2}_xf(y)=\frac{1}{2}\int_{-1}^1f\Bigl(\sqrt{x^2+y^2+2xyt}\Bigr)\biggl(1+\frac{x+y}{\sqrt{x^2+y^2+2xyt}}\biggr)\psi_\kappa(t)\,\mathrm{d}t \\
+\frac{1}{2}\int_{-1}^1f\Bigl(-\sqrt{x^2+y^2+2xyt}\Bigr)\biggl(1-\frac{x+y}{\sqrt{x^2+y^2+2xyt}}\biggr)\psi_\kappa(t)\,\mathrm{d}t
\end{multline}
where $\psi$ is given by $\psi_\kappa(t)=\bigl(B\bigl(\kappa,\frac{1}{2}\bigr)\bigr)^{-1}(1+t)(1-t^2)^{\kappa-1}$ (with $B$ the beta function) and where $\kappa$ is the only value taken by the multiplicity function $\kappa$. This formula implies an explicit one in the case $\mathbb Z_2^d$ which gives us the boundedness of the Dunkl translation. In order to give an equivalent formula to \eqref{extau}, we need to introduce some notation.

\begin{nota}
\mbox{}
\begin{enumerate} 
\item
For $x,y,z \in \mathbb R$, we put 
\[\sigma_{x,y,z}=
\begin{cases}
\frac{1}{2xy}(x^2+y^2-z^2)& \text{if } x,y \neq 0, \\
0& \text{if } x=0 \text{ or } y=0,
\end{cases}
\]
as well as
\[
\varrho(x,y,z)=\frac{1}{2}(1-\sigma_{x,y,z}+\sigma_{z,x,y}+\sigma_{z,y,x}).
\]
\item
For $x,y,z>0$, we put
\[
K_\kappa(x,y,z)=2^{2\kappa-2}\Bigl(B\Bigl(\kappa,\frac{1}{2}\Bigr)\Bigr)^{-1}\frac{\Delta(x,y,z)^{2\kappa-2}}{(xyz)^{2\kappa-1}}\chi_{{}_{[|x-y|,x+y]}}(z),
\]
where  $\Delta(x,y,z)$ denotes the area of the triangle (perhaps degenerated) with sides $x,y,z$.
\end{enumerate}
\end{nota} 
 
With this notation, \eqref{extau} can be reformulated (using a change of variables) as follows:
\begin{equation} \label{extaubis}
\tau^{\mathbb Z_2}_xf(y)=\int_{\mathbb R}f(z)\,\mathrm{d}\nu^{\mathbb Z_2}_{x,y}(z)
\end{equation}
where the measure $\nu^{\mathbb Z_2}_{x,y}$ is given by
\[
\mathrm d\nu^{\mathbb Z_2}_{x,y}(z)=
\begin{cases}
\mathcal K_\kappa(x,y,z)\,\mathrm{d}\mu^{\mathbb Z_2}_\kappa(z)& \text{if } x,y \neq 0, \\
\mathrm d\delta_x(z)& \text{if } y=0, \\
\mathrm d\delta_y(z)& \text{if } x=0,
\end{cases}
\]
with 
\[
\mathcal K_\kappa(x,y,z)=K_\kappa\bigl(|x|,|y|,|z|\bigr)\varrho(x,y,z).
\]
We have the following one-dimensional result due to  Rösler, .

\begin{theo} \label{mesros}
The measure $\nu^{\mathbb Z_2}_{x,y}$ satisfies
\begin{enumerate}
\item $\mathrm{supp}\,\nu^{\mathbb Z_2}_{x,y}=\Bigl[-|x|-|y|,-\bigl||x|-|y|\bigr|\Bigr]\bigcup\Bigl[\bigl||x|-|y|\bigr|,|x|+|y|\Bigr]$ for $x, y \neq 0$.
\item $\nu^{\mathbb Z_2}_{x,y}(\mathbb R)=1$ and $\|\nu^{\mathbb Z_2}_{x,y}\|=\int_{\mathbb R}|\mathrm d\nu^{\mathbb Z_2}_{x,y}|\leqslant 4$, for $x, y \in \mathbb R$.
\end{enumerate}
\end{theo}
We now prove that no analogue of Theorem \ref{FS}$(2)$ holds when $p=+\infty$.

\begin{pr}
The conclusion of the second point of Theorem \ref{FS} does not hold when $p=+\infty$.
\end{pr}

\begin{proof}
Let $d=1$ and let $f=(f_n)_{n\geqslant1}$ where  $f_n=\chi_{{}_{[2^{n-1},2^n[}}$ for every $n\geqslant1$. We easily see that 
\[
\|f\|_{\ell^r}=\chi_{{}_{[1,+\infty[}} \in L^\infty(\mu_\kappa^{\mathbb Z_2}) 
\]
while we will prove that
\begin{equation} \label{linfty} 
\|\mathcal M^{\mathbb Z_2}_\kappa f\|_{\ell^r} \notin L^\infty(\mu_\kappa^{\mathbb Z_2}).
\end{equation} 
For all $x \in \mathbb R$ and  $n\geqslant 1$,  by definition of $M^{\mathbb Z_2}_\kappa\chi_{{}_{[2^{n-1},2^n[}}$,
\begin{multline*}
M^{\mathbb Z_2}_\kappa\chi_{{}_{[2^{n-1},2^n[}}(x)\\
\geqslant\frac{1}{\mu_\kappa^{\mathbb Z_2}(]-|x|-2^{n},|x|+2^n[)}\int_{2^{n-1}}^{2^n}\tau^{\mathbb Z_2}_x(\chi_{{}_{]-|x|-2^n,|x|+2^n[}})(-y)\,\mathrm{d}\mu_\kappa^{\mathbb Z_2}(y). 
\end{multline*}
We claim that for every $y \in [2^{n-1},2^n[$,
\begin{equation} \label{tau1}
\tau^{\mathbb Z_2}_x(\chi_{{}_{]-|x|-2^n,|x|+2^n[}})(-y)=1. 
\end{equation}
Indeed,  by using \eqref{extaubis} we can write
\begin{multline*}
2^{2-2\kappa}B\Bigl(\kappa,\frac{1}{2}\Bigr)\tau^{\mathbb Z_2}_x(\chi_{{}_{]-|x|-2^n,|x|+2^n[}})(-y)\\
=\int_{\mathbb R}\chi_{{}_{]-|x|-2^n,|x|+2^n[}}(z)\chi_{{}_{]||x|-|y||,|x|+|y|[}}(|z|)\frac{\Delta(|x|,|y|,|z|)^{2\kappa-2}}{(|xyz|)^{2\kappa-1}}\varrho{(x,-y,z)}\,\mathrm{d}\mu_\kappa^{\mathbb Z_2}(z),
\end{multline*}
and, since $|x|+|y|<|x|+2^{n}$, we  deduce that
\begin{multline*}
 2^{2-2\kappa}B\Bigl(\kappa,\frac{1}{2}\Bigr)\tau^{\mathbb Z_2}_x(\chi_{{}_{]-|x|-2^n,|x|+2^n[}})(-y)\\
=\int_{\mathbb R}\chi_{{}_{]||x|-|y||,|x|+|y|[}}(|z|)\frac{\Delta(|x|,|y|,|z|)^{2\kappa-2}}{(|xyz|)^{2\kappa-1}}\varrho{(x,-y,z)}\,\mathrm{d}\mu_\kappa^{\mathbb Z_2}(z),
\end{multline*}
that is,
\[
 \tau^{\mathbb Z_2}_x(\chi_{{}_{]-|x|-2^n,|x|+2^n[}})(-y)=\nu^{\mathbb Z_2}_{x,-y}(\mathbb R).
\]
We obtain \eqref{tau1} by using Theorem \ref{mesros}.
As a result we have the inequality
\[
M^{\mathbb Z_2}_\kappa\chi_{{}_{[2^{n-1},2^n[}}(x)\geqslant\frac{1}{\mu_\kappa^{\mathbb Z_2}(]-|x|-2^{n},|x|+2^n[)}\int_{2^{n-1}}^{2^n}\mathrm{d}\mu_\kappa^{\mathbb Z_2}(y),
\]
and since $\mathrm{d}\mu_\kappa^{\mathbb Z_2}(y)=|y|^{2\kappa}\mathrm{d}y$, we are led to
\[
 M^{\mathbb Z_2}_\kappa\chi_{{}_{[2^{n-1},2^n[}}(x)\geqslant\frac{2^{n(2\kappa+1)}-2^{(n-1)(2\kappa+1)}}{2(|x|+2^{n})^{2\kappa+1}}.
\]
For every $x$ satisfying  $|x|\leqslant2^{n}$ we thus obtain
\[
 M^{\mathbb Z_2}_\kappa\chi_{{}_{[2^{n-1},2^n[}}(x)\geqslant\frac{2^{n(2\kappa+1)}-2^{(n-1)(2\kappa+1)}}{2^{(n+1)(2\kappa+1)+1}},
\]
and after simplifications  we get 
\[
M^{\mathbb Z_2}_\kappa\chi_{{}_{[2^{n-1},2^n[}}(x)\geqslant\frac{1-2^{-(2\kappa+1)}}{2^{2\kappa+2}}.
\]
Thus,  for every $x \in \mathbb R$,
\[
\|\mathcal M^{\mathbb Z_2}_\kappa f\|_{\ell^r}^r\geqslant\sum_{\{n:2^{n}\geqslant|x|\}}\Bigl(\frac{1-2^{-(2\kappa+1)}}{2^{2\kappa+2}}\Bigr)^r=+\infty,
\]
and the proof of \eqref{linfty} is finished.
\end{proof}

A substitute result in the case  $p=+\infty$ is  given in Theorem \ref{expint} that we now prove. In order to do that, we need three lemmas. The first one is just a trivial functional equality.

\begin{lem} \label{ie1}
 Let $X$ be a measure space and let $m$ be a positive measure on $X$. Let $\varphi$ be an increasing continuously differentiable function on $[0,+\infty[$ which satisfies $\varphi(0)=0$. Then for every measurable function $f$ on $X$, 
\[
 \int_X\varphi\bigl(|f(x)|\bigr)\,\mathrm{d}m(x)=\int_0^{+\infty}\varphi'(\lambda)m\Bigl(\Bigl\{x \in X: |f(x)|>\lambda\Bigr\}\Bigr)\,\mathrm{d}\lambda.
\]
\end{lem}

\begin{proof}
 Since $\varphi'$ is nonnegative, the Fubini theorem yields
\[
\int_0^{+\infty}\varphi'(\lambda)m\Bigl(\Bigl\{x \in X: |f(x)|>\lambda\Bigr\}\Bigr)\,\mathrm{d}\lambda=\int_X\biggl(\int_0^{|f(x)|}\varphi'(\lambda)\,\mathrm{d}\lambda\biggr)\,\mathrm{d}m(x). 
\]
 The desired result follows by integrating and using the fact that $\varphi(0)=0$.
\end{proof}

The second lemma gives us the behavior of the constant of  Theorem \ref{FS}$(2)$ when $p$ grows and $r$ is fixed. The result of exponential integrability is closely related to this behavior.

\begin{lem} \label{ie2}
The constant $C(d,\kappa,p,r)$ of the vector-valued maximal theorem is such that
\[
C(d,\kappa,p,r)=\underset{p\to+\infty}{O}\bigl(p^{\frac{1}{r}}\bigr). 
\]
 \end{lem}

\begin{proof} 
 Since the parameter $r$ is fixed, it is enough to consider the  proof of Theorem \ref{FS}  when $p>r$. As explained in \cite{Luc}, once we have constructed the operator $M_\kappa^{\mathbb{Z}_2^d,R}$ (see \cite{Luc} for the definition) and shown that it satisfies a scalar maximal theorem and a weighted inequality, we can follow almost verbatim the proof of the vector-valued maximal theorem for the Hardy-Littlewood maximal operator  (see \cite{FeffermanStein} or \cite{SteinGros}). Thus, it is easily seen that the dependence on $p$ is given by the constant (with exponent $\tfrac{1}{r}$) of the maximal theorem for  $M_\kappa^{\mathbb{Z}_2^d,R}$ and for the space $L^{\frac{p}{p-r}}(\mu_\kappa^{\mathbb{Z}_2^d})$. Since this constant is obtained by interpolation, we can  write 
\[
 C(d,\kappa,p,r)=C(d,\kappa,r)\biggl(C(d,\kappa)\frac{\frac{p}{p-r}}{\frac{p}{p-r}-1}\biggr)^{\frac{p-r}{pr}}.
\]
But it is obvious that
\[
 \biggl(C(d,\kappa)\frac{\frac{p}{p-r}}{\frac{p}{p-r}-1}\biggr)^{\frac{p-r}{p}}=\underset{p\to+\infty}{O}\bigl(p\bigr),
\]
and the lemma is proved.
\end{proof}
 
The last lemma gives us (under the hypothesis of Theorem \ref{expint}) a sharp estimate of the measure of the set $\{x \in K: \|\mathcal{M}_\kappa^{\mathbb{Z}_2^d}f(x)\|_{\ell^r}^r>\lambda\}$ for every $\lambda>0$, where $K$ denotes a compact subset of $\mathbb R^d$. More precisely we have the following inequality.

\begin{lem} \label{ie3}
Let $f=(f_n)_{n\geqslant1}$ be a sequence of measurable functions defined on $\mathbb{R}^d$ and let $r$ satisfy $1<r<+\infty$. If  $\|f\|_{\ell^r} \in L^\infty(\mu^{\mathbb{Z}^d_2}_\kappa)$ is such that
\[
\mu^{\mathbb{Z}^d_2}_\kappa\Bigl(\mathrm{supp }\|f\|_{\ell^r}^r\Bigr)<+\infty, 
\]
then for every compact subset $K$ of $\mathbb{R}^d$ and every $\lambda>0$,
\begin{multline*}
 \mu^{\mathbb{Z}^d_2}_\kappa\Bigl(\Bigl\{x \in K: \|\mathcal{M}_\kappa^{\mathbb{Z}_2^d}f(x)\|_{\ell^r}^r>\lambda\Bigr\} \Bigr)\\
\leqslant \max\Bigl\{2\mu^{\mathbb{Z}^d_2}_\kappa(K);\mu^{\mathbb{Z}^d_2}_\kappa\Bigl(\mathrm{supp }\|f\|_{\ell^r}^r\Bigr)\Bigr\}\mathrm{e}^{-\Bigl(\frac{\log (2)}{2C\|\|f\|_{\ell^r}^r\|_{\mathbb{Z}_2^d,\infty}}\Bigr)\lambda},
\end{multline*}
where $C=C(d,\kappa,r)$ is independent of $(f_n)_{n\geqslant1}$, $K$ and $\lambda$.
\end{lem}

\begin{proof}
For every $p$ satisfying $1\leqslant p<+\infty$, the Chebyshev inequality yields
\begin{multline*} 
 \mu^{\mathbb{Z}^d_2}_\kappa\Bigl(\Bigl\{x \in K: \|\mathcal{M}_\kappa^{\mathbb{Z}_2^d}f(x)\|_{\ell^r}^r>\lambda\Bigr\} \Bigr)\\
\leqslant\frac{1}{\lambda^p}\int_{\bigl\{x \in K: \, \|\mathcal{M}_\kappa^{\mathbb{Z}_2^d}f(x)\|_{\ell^r}^r>\lambda\bigr\}}\|\mathcal M^{\mathbb{Z}^d_2}_\kappa f(x)\|_{\ell^r}^{pr}\,\mathrm{d}\mu^{\mathbb{Z}^d_2}_\kappa(x), 
\end{multline*}
which implies, by enlarging the domain of integration,
\[
 \mu^{\mathbb{Z}^d_2}_\kappa\Bigl(\Bigl\{x \in K: \|\mathcal{M}_\kappa^{\mathbb{Z}_2^d}f(x)\|_{\ell^r}^r>\lambda\Bigr\} \Bigr)\leqslant \frac{1}{\lambda^p}\int_{\mathbb{R}^d}\|\mathcal M^{\mathbb{Z}^d_2}_\kappa f(x)\|_{\ell^r}^{pr}\,\mathrm{d}\mu^{\mathbb{Z}^d_2}_\kappa(x).
\]
By applying the vector-valued maximal theorem for $\mathcal M^{\mathbb{Z}^d_2}_\kappa$ we get
\[
 \mu^{\mathbb{Z}^d_2}_\kappa\Bigl(\Bigl\{x \in K: \|\mathcal{M}_\kappa^{\mathbb{Z}_2^d}f(x)\|_{\ell^r}^r>\lambda\Bigr\} \Bigr)\leqslant \frac{\bigl(C(d,\kappa,pr,r)\bigr)^{pr}}{\lambda^p}\int_{\mathbb{R}^d}\|f(x)\|_{\ell^r}^{pr}\,\mathrm{d}\mu^{\mathbb{Z}^d_2}_\kappa(x),
\]
where $C=C(d,\kappa,pr,r)$ is the constant of Theorem \ref{FS} (and which is independent of $(f_n)_{n\geqslant1}$, $K$ and $\lambda$). Thanks to Lemma \ref{ie2} we are lead to
\[
 \mu^{\mathbb{Z}^d_2}_\kappa\Bigl(\Bigl\{x \in K: \|\mathcal{M}_\kappa^{\mathbb{Z}_2^d}f(x)\|_{\ell^r}^r>\lambda\Bigr\} \Bigr)\leqslant \frac{Cp^p}{\lambda^p}\int_{\mathbb{R}^d}\|f(x)\|_{\ell^r}^{pr}\,\mathrm{d}\mu^{\mathbb{Z}^d_2}_\kappa(x), 
\]
where $C=C(d,\kappa,r)$ is independent of $(f_n)_{n\geqslant1}$,  $K$,  $\lambda$ and $p$. The hypothesis of the lemma allows us to write
\begin{multline} \label{ie3eq1}
 \mu^{\mathbb{Z}^d_2}_\kappa\Bigl(\Bigl\{x \in K: \|\mathcal{M}_\kappa^{\mathbb{Z}_2^d}f(x)\|_{\ell^r}^r>\lambda\Bigr\} \Bigr)\\
\leqslant \biggl(\frac{Cp\,\|\|f\|_{\ell^r}^r\|_{\mathbb{Z}_2^d,\infty}}{\lambda}\biggr)^p\mu^{\mathbb{Z}^d_2}_\kappa\Bigl(\mathrm{supp }\|f\|_{\ell^r}^r\Bigr).
\end{multline}
We now exploit \eqref{ie3eq1} by choosing $p$ in terms of $\lambda$.  If $\lambda\geqslant 2C\|\|f\|_{\ell^r}^r\|_{\mathbb{Z}_2^d,\infty}$, we put
\[
 p=\frac{\lambda}{2C\|\|f\|_{\ell^r}^r\|_{\mathbb{Z}_2^d,\infty}}\geqslant1.
\]
Then \eqref{ie3eq1} can be reformulated as follows:
\[
 \mu^{\mathbb{Z}^d_2}_\kappa\Bigl(\Bigl\{x \in K: \|\mathcal{M}_\kappa^{\mathbb{Z}_2^d}f(x)\|_{\ell^r}^r>\lambda\Bigr\} \Bigr)\leqslant\Bigl(\frac{1}{2}\Bigr)^p\mu^{\mathbb{Z}^d_2}_\kappa\Bigl(\mathrm{supp }\|f\|_{\ell^r}^r\Bigr), 
\]
that is,
\begin{equation} \label{ie3eq2}
  \mu^{\mathbb{Z}^d_2}_\kappa\Bigl(\Bigl\{x \in K: \|\mathcal{M}_\kappa^{\mathbb{Z}_2^d}f(x)\|_{\ell^r}^r>\lambda\Bigr\} \Bigr)
\leqslant
\mathrm{e}^{-\Bigl(\frac{\log (2)}{2C\|\|f\|_{\ell^r}^r\|_{\mathbb{Z}_2^d,\infty}}\Bigr)\lambda}\mu^{\mathbb{Z}^d_2}_\kappa\Bigl(\mathrm{supp }\|f\|_{\ell^r}^r\Bigr).
\end{equation}
 If $\lambda\leqslant 2C\|\|f\|_{\ell^r}^r\|_{\mathbb{Z}_2^d,\infty}$, we immediately obtain
\[
 \mu^{\mathbb{Z}^d_2}_\kappa\Bigl(\Bigl\{x \in K: \|\mathcal{M}_\kappa^{\mathbb{Z}_2^d}f(x)\|_{\ell^r}^r>\lambda\Bigr\} \Bigr)\leqslant \mu^{\mathbb{Z}^d_2}_\kappa(K).
\]
Since
\[
 \mathrm{e}^{-\Bigl(\frac{\log (2)}{2C\|\|f\|_{\ell^r}^r\|_{\mathbb{Z}_2^d,\infty}}\Bigr)\lambda}\geqslant\frac{1}{2},
\]
we find that
\begin{equation} \label{ie3eq3}
 \mu^{\mathbb{Z}^d_2}_\kappa\Bigl(\Bigl\{x \in K: \|\mathcal{M}_\kappa^{\mathbb{Z}_2^d}f(x)\|_{\ell^r}^r>\lambda\Bigr\} \Bigr)\leqslant 2\mu^{\mathbb{Z}^d_2}_\kappa(K) \mathrm{e}^{-\Bigl(\frac{\log (2)}{2C\|\|f\|_{\ell^r}^r\|_{\mathbb{Z}_2^d,\infty}}\Bigr)\lambda}.
\end{equation}
The desired result is then a trivial consequence of  \eqref{ie3eq2} and \eqref{ie3eq3}.
\end{proof}

We are now in a position to prove Theorem \ref{expint}.

\begin{proof}
 Let $K$ be a compact subset of $\mathbb{R}^d$ and let $\varepsilon$ be a real number which satisfies
\[
0\leqslant\varepsilon<\frac{\log( 2)}{2C_{d,\kappa,r}\bigl\|\,\|f\|_{\ell^r}^r\bigr\|_{\mathbb{Z}_2^d,\infty}}, 
\]
where $C_{d,\kappa,r}$ is the constant of Lemma \ref{ie3}. We first write
\[
\int_K\mathrm{e}^{\varepsilon\|\mathcal{M}_\kappa^{\mathbb{Z}_2^d}f(x)\|_{\ell^r}^r}\,\mathrm{d}\mu_\kappa^{\mathbb{Z}_2^d}(x)=\mu_\kappa^{\mathbb{Z}_2^d}(K)+\int_K\Bigl(\mathrm{e}^{\varepsilon\|\mathcal{M}_\kappa^{\mathbb{Z}_2^d}f(x)\|_{\ell^r}^r}-1\Bigr)\,\mathrm{d}\mu_\kappa^{\mathbb{Z}_2^d}(x). 
\]
We then apply the equality of Lemma \ref{ie1} to the function $\varphi:t\mapsto \mathrm{e}^{\varepsilon t}-1$ to obtain
\begin{multline*}
\int_K\mathrm{e}^{\varepsilon\|\mathcal{M}_\kappa^{\mathbb{Z}_2^d}f(x)\|_{\ell^r}^r}\,\mathrm{d}\mu_\kappa^{\mathbb{Z}_2^d}(x)\\
=\mu_\kappa^{\mathbb{Z}_2^d}(K)+\varepsilon\int_0^{+\infty}\mathrm{e}^{\varepsilon\lambda}\mu^{\mathbb{Z}^d_2}_\kappa\Bigl(\Bigl\{x \in K: \|\mathcal{M}_\kappa^{\mathbb{Z}_2^d}f(x)\|_{\ell^r}^r>\lambda\Bigr\} \Bigr)\,\mathrm{d}\lambda. 
\end{multline*}
Thanks to Lemma \ref{ie3}, we are led to
\begin{multline*}
\int_K\mathrm{e}^{\varepsilon\|\mathcal{M}_\kappa^{\mathbb{Z}_2^d}f(x)\|_{\ell^r}^r}\,\mathrm{d}\mu_\kappa^{\mathbb{Z}_2^d}(x)\\ \leqslant\mu_\kappa^{\mathbb{Z}_2^d}(K)+\varepsilon\max\Bigl\{2\mu^{\mathbb{Z}^d_2}_\kappa(K);\mu^{\mathbb{Z}^d_2}_\kappa\Bigl(\mathrm{supp }\|f\|_{\ell^r}^r\Bigr)\Bigr\}\int_0^{+\infty}\mathrm{e}^{\lambda\Bigl(\varepsilon-\frac{\log (2)}{2C\|\|f\|_{\ell^r}^r\|_{\mathbb{Z}_2^d,\infty}}\Bigr)}\,\mathrm{d}\lambda. 
\end{multline*}
The condition on $\varepsilon$ and an integration allow us to conclude.
\end{proof}
\bibliography{tworesults}
\end{document}